\documentclass[a4paper]{article}
\usepackage[utf8]{inputenc}
\usepackage[T1]{fontenc}

\usepackage{amsmath,amsfonts,amssymb}
\usepackage{mathtools}
\usepackage{amsthm}
\usepackage[]{authblk} 
\usepackage[svgnames]{xcolor}
\usepackage[]{geometry}

\usepackage[]{hyperref}
\hypersetup{colorlinks=true,
  citecolor=DarkBlue,
  linkcolor=DarkBlue,
  linktoc=page,
  urlcolor=black,
}

\title{Entropy flows and functional inequalities in convex sets}
\date{}

\author{Simon Zugmeyer}

\affil{Univ Lyon, Université Claude Bernard Lyon 1, CNRS UMR 5208, Institut Camille Jordan, 43 blvd. du 11 novembre 1918, F-69622 Villeurbanne cedex, France. \texttt{\href{mailto:zugmeyer@math.univ-lyon1.fr}{zugmeyer@math.univ-lyon1.fr}}}

\newcommand{\titlemath}[2]{\texorpdfstring{\boldmath#1}{#2}}

\newtheorem{thrm}{Theorem}[section]
\newtheorem{prop}[thrm]{Proposition}
\newtheorem{coro}[thrm]{Corollary}
\newtheorem{lemm}[thrm]{Lemma}

\theoremstyle{definition}
\newtheorem{defi}[thrm]{Definition}

\theoremstyle{remark}
\newtheorem{rmrk}[thrm]{Remark}
\newtheorem{hypo}{Hypothesis}

\newcommand{\R}{\mathbb{R}}

\newcommand{\N}{\mathbb{N}}

\newcommand{\vphi}{\varphi}
\newcommand{\eps}{\varepsilon}
\newcommand{\mc}{\mathcal}

\DeclareMathOperator{\Ent}{Ent}

\DeclareMathOperator{\supp}{supp}
\DeclareMathOperator{\dive}{div}

\DeclarePairedDelimiter{\norm}{\lVert}{\rVert} 
\DeclarePairedDelimiter{\abs}{\lvert}{\rvert}
\DeclarePairedDelimiter{\sbra}{[}{]}
\DeclarePairedDelimiter{\pare}{(}{)}

\makeatletter
\let\oldabs\abs
\def\abs{\@ifstar{\oldabs}{\oldabs*}}
\let\oldnorm\norm
\def\norm{\@ifstar{\oldnorm}{\oldnorm*}}
\let\oldsbra\sbra
\def\sbra{\@ifstar{\oldsbra}{\oldsbra*}}
\let\oldpare\pare
\def\pare{\@ifstar{\oldpare}{\oldpare*}}
\makeatother

\begin{document}

\maketitle

\begin{abstract}
  We revisit entropy methods to prove new sharp trace logarithmic Sobolev and sharp Gagliardo-Nirenberg-Sobolev inequalities on the half space, with a focus on the entropy inequality itself and not the actual flow, allowing for somewhat robust and self-contained proofs. 
  
\end{abstract}

\section{Introduction}\label{se:intro}

\subsection{Brief introduction of the ideas}

Sobolev inequalities have proved an important tool in the study of partial differential equations, notably in establishing existence results. More recently, they have been very fruitfully used in the study of the long term behavior of certain equations. For instance, the logarithmic Sobolev inequality can be used to establish a rate of convergence of the heat flow towards its mean on torus, or towards the self-similar profile on the Euclidean space.

These ideas and results fall in the general context of entropy methods, of which Ansgar Jüngel's book~\cite{jue_book} offers a nice overview. Boltzmann defined his entropy in 1872 by \(\mc \Ent(u)=\int u\log(u)dx\) for, say, positive functions defined on \(\R^d\). Now, if \(u\) is the solution of the heat equation, \(\partial_t u=\Delta u\), differentiating the entropy yields
\[
\frac{d}{dt}\Ent(u)=\int(1+\log(u))\Delta u dx = -\int \frac{\norm{\nabla u}^2}{u}dx,
\]
which implies two important things: the first one is that the entropy is nonincreasing along the flow, conveying the idea that the physical transformation described by \(u\) is irreversible. The second one, maybe more profound, is that the logarithmic Sobolev inequality is exactly an equality relating the entropy and its derivative, which, after integration, implies exponential decay of the entropy with respect to time. This relationship between the heat equation and the entropy is no coincidence. It turns out that the space of probability measures, when equipped with the Wasserstein distance, can be formally seen as a Riemannian manifold, and in that setting, the heat equation is exactly the gradient flow of Boltzmann's entropy. This approach was initially developed in the seminal papers by Felix Otto et al.~\cite{jko98,ov00,otto01}, and the study of gradient flows on metric spaces has since been made rigorous in~\cite{ags_book}. See also Filippo Santambrogio's survey on this topic~\cite{santambrogio_survey}. 

Somewhat recently, Manuel del Pino and Jean Dolbeault observed a similar result for the Gagliardo-Nirenberg-Sobolev (GNS) inequalities \cite{dpd99,dpd02}. They may be rewritten in a way that involves both an entropy functional, different than Boltzmann's, and its derivative along a particular mass-preserving flow.
\begin{thrm}[{\cite[Corollary 13]{dpd02}}]\label{th:dpd_gns}
  Let \(d\geq 3\), and \(1<p<\frac{d}{d-2}\). Then, the sharp Gagliardo-Nirenberg-Sobolev inequality
  \begin{equation}\label{eq:dpd_gns}
    \forall w\in\mc C^\infty_c(\R^d),\,\norm{w}_{2p}\leq C \norm{\nabla w}_2^\theta \norm{w}_{1+p}^{1-\theta},
  \end{equation}
  where \(\theta\in\sbra{0,1}\) is fixed by the parameters, is equivalent to
  \begin{equation}\label{eq:dpd_entropy}
    \forall u\in\mc C^\infty_c(\R^d) \text{ s.t.}\  \norm u_1=\norm v_1,\, \mc F(u)-\mc F(v) \leq \frac{1}{2}\int_{\R^d}u\norm{x-\nabla(u^{\alpha-1})}^2dx,
  \end{equation}
  where \(\alpha=(p+1)/(2p)<1\), and
  \[
  \mc F(u)=\int_{\R^d} \sbra{-\frac{u^\alpha}{\alpha} + u\pare{1+\frac{\norm{x}}{2}^2}}dx \quad \text{and}\quad  v(x)=\pare{1+\frac{\norm{x}}{2}^2}^{1/(\alpha-1)}.
  \]
\end{thrm}
Again, the right-hand side in inequality \eqref{eq:dpd_entropy} is exactly the derivative of the functional \(\mc F\) along a flow obtained, depending on the value of \(\alpha\), from the porous medium equation or the fast-diffusion equation through a change of variables. The resulting exponential decay of the entropy \(\mc F\) may be finally used to prove decay in various \(L^p\) norms using a general Csisz\'ar-Kullback-Pinsker inequality \cite{amtu01}. Note even though \(\mc F\) is not Boltzmann's original entropy, we still call it entropy, and we will use that word again throughout the paper without a rigorous definition. Loosely speaking, entropies will be Lyapunov functional related to specific flows, but the important thing is that they are only a tool. Many examples of entropies used to prove long-time behavior of solutions to certain equations exist in the literature, see for example~\cite{ct00,cv03}, and the references in~\cite{jue_book}.

Logarithmic Sobolev and Gagliardo-Nirenberg-Sobolev inequalities are not the only ones that may be seen as an inequality between the entropy and its derivative along a flow. Indeed, Poincaré inequalities, and the so-called Beckner inequalities, interpolating inequalities between the Poincaré inequality and the logarithmic Sobolev inequality, are all natural examples of this \cite{GZ19}. This motivates the study of the so-called generalized Sobolev inequalities.

\bigskip

While these generalized Sobolev inequalities may be used to study a particular flow, a striking fact is that they turn out to be contained in the flow itself~\cite{cjmtu01,jue_book}. Indeed, in the good tradition of the Bakry-Émery method \cite{be85}, differentiating the entropy along the flow twice instead of once, and then invoking geometric properties of the underlying space (the Bochner-Lichnerowicz inequality, or a curvature-dimension condition) as well as the convexity inherent of the entropy functional, allows to recover said generalized Sobolev inequality. This method was succesfully used by Toscani for the logarithmic Sobolev inequality \cite{tos97}, and subsequently quite thoroughly investigated in \cite{cjmtu01}. While optimal transport was used for a short proof of the Sobolev inequality in~\cite{cnv04}, flows have the advantage of being easily generalized to manifolds, as Demange has cleverly done in~\cite{dem08}.

In the case of linear flows, the existence of Markov semigroups makes the study simpler~\cite{bgl_book}, but in the general case, one has to resort to use tools from the realm of partial differential equations, possibly making the study quite convoluted. In this article, we wish to revisit the method for general entropies of the form
\begin{equation}\label{eq:entropy_ex}
  \mc F(u)=\int H(u)+uV,
\end{equation}
where \(H\) is a convex function on \(\R_+\), and \(V\) a strictly uniformly convex function on some subdomain of \(\R^d\).
\begin{rmrk}
  In equation \eqref{eq:entropy_ex} much like in the rest of the article, the integration against the standard Lebesgue measure, or the \((d-1)\)-dimensional Hausdorff measure when integrating on boundaries, will always be implied.
\end{rmrk}
We try for the proofs to be as self-contained as possible, while also keeping the calculations to a minimum. This is possible since we do not want to study the long-term behavior of the various flows considered, and are only interested in proving generalized Sobolev inequalities. Furthermore, we will, starting in section \ref{se:heart}, use the same vocabulary that is used in \cite{bgl_book}, i.e.\ we will make use of the \emph{Carré du champ} operator \(\Gamma\) and its iterated version, \(\Gamma_2\). This choice is motivated by two reasons: the first one is because the main results are very similar in nature with ones involving Markov semigroups.
The second one is because it makes calculations systematic, and also makes the curvature-dimension hypotheses appear clearly, allowing for easy generalization of all the results to manifolds.


\subsection{Model example: the Euclidean Sobolev inequality}\label{se:model}

Let us showcase the method with the study of a simple example which will serve as a guide in the next sections: the proof of the sharp Sobolev inequality on \(\R^d\). Define the functions \(H\in C^\infty(\R_+^*,\R)\) and \(v\in C^\infty(\R^d,\R_+^*)\) by
\begin{equation}\label{eq:fd_fun}
  H(x)=-x^{1-1/d}, \quad\text{and}\quad v(x)=C(1+\norm x^2)^{-d},
\end{equation}
where \(C>0\) has been chosen so that \(-H'(v)=1+\norm{x}^2\). For smooth positive functions \(u\), we define the entropy
\begin{equation}\label{eq:entropydef}
  \mc F(u) = \int_{\R^d} H(u) -u H'(v).
\end{equation}
Note that, since \(H\) is convex, \(\mc F(u)\geq \mc F(v)\).
Now choose a function \(u_0\in C^\infty(\R^d,\R_+^*)\) such that \(\int_{\R^d} u_0=\int_{\R^d}v \), and consider the relative entropy
\[
\mc F(u\mid v) = \mc F(u) -\mc F(v)
\]
along the flow
\begin{equation}\label{eq:gradflow}
  \begin{aligned}
    &\partial_tu=-\nabla\cdot\big(u\nabla(H'(v)-H'(u))\big) && \quad\text{in}\  \R_+^*\times\R^d, \\
    &u(0,.)=u_0 &&\quad \text{in}\  \R^d.
  \end{aligned}
\end{equation}
The first derivative of the entropy is easily calculated using an integration by parts:
\begin{align*}
  \frac{d}{dt}\mc F(u)&=\int_{\R^d} \partial_tu(H'(u)-H'(v))\\
  &=\int_{\R^d} \nabla\cdot(u\nabla(H'(v)-H'(u)))(H'(v)-H'(u))\\
  &=-\int_{\R^d} u\norm{\nabla(H'(v)-H'(u))}^2 \leq 0.
\end{align*}
Note that the flow \eqref{eq:gradflow} is the gradient flow of the entropy functional \eqref{eq:entropydef}. The fact that the derivative of the entropy takes such a nice form is a general fact of gradient flows \cite{santambrogio_survey}. The calculations for the second derivative are slightly tricky, so we refer to the next section for the full details, but using both the fact that \(\nabla^2H'(v)=-2 I_d\) and that \(\norm{\nabla^2\phi}_{HS}^2\geq \frac{1}{d}(\Delta\phi)^2\), we find that 
\begin{equation}\label{eq:entropy_ineq1} 
  \frac{d^2}{dt^2}\mc F(u) \geq -4 \frac{d}{dt}\mc F(u),
\end{equation}
which, if one recalls that the first derivative is nonpositive, proves that the entropy along the flow has a strong convexity property which is really the core of the argument.
Assuming that the function \(u\) converges, when \(t\) goes to infinity, to the stationnary solution \(v\), it is quite clear that \(\lim_{t\to+\infty}\mc F(u)=\mc F(v)\), and \(\lim_{t\to+\infty}\frac{d}{dt}\mc F(u)=0\). Now, integrating the second derivative of the entropy between \(0\) and \(+\infty\) leads to
\begin{equation}\label{eq:entropy_entropy_prod}
  \mc F(u_0) - \lim_{t\to+\infty}\mc F(u)=\mc F(u_0\mid v)\leq - \frac{1}{4} \left.\frac{d}{dt}\mc F(u)\right\vert_{t=0}.
\end{equation}
Equation \eqref{eq:entropy_entropy_prod} is a special case of an entropy - entropy production inequality, to which we will come back later. It is quite obviously optimal, since equality happens for \(u_0=v\).

We may now rewrite equation \eqref{eq:entropy_entropy_prod} with the explicit quantities \eqref{eq:fd_fun} to prove the sharp Sobolev inequality on \(\R^d\): since \(-H'(v)=1+\norm x^2\),
\begin{equation}
  \int_{\R^d} H(u_0)-H(v)-(u_0-v)H'(v)\leq \frac{1}{4}\int_{\R^d} u_0\norm{\nabla H'(u_0)+2x}^2.
\end{equation}
Expanding the right-hand side, we have to deal with three different terms. First, notice that
\[
\frac{1}{4}\int_{\R^d} u_0\norm{\nabla H'(u_0)}^2 = \frac{(d-1)^2(d-2)^2}{16d^2}\int_{\R^d} \norm{\nabla u_0^{1/2-1/d}}^2
\]
Next, the other square is
\[
\int_{\R^d} u_0\norm x^2 = \int_{\R^d} u_0(-H'(v)-1),
\]
which simplifies with the left-hand side. Finally, the double product can be integrated by parts once we notice, once again by homogeneity, that \(u_0\nabla H'(u_0)=-\frac{1}{d}\nabla H(u_0)\):
\[
\int_{\R^d} u_0\nabla H'(u_0)\cdot x = -\frac{1}{d}\int_{\R^d} \nabla H(u_0)\cdot x = \int_{\R^d} H(u_0),
\]
and this also simplifies with the left-hand side. Since \(\int u_0=\int v\), the equation we are left with is
\[
C \leq \int_{\R^d} \norm{\nabla u_0^{1/2-1/d}}^2
\]
for some explicit positive constant \(C\). Replacing \(u_0\) with \(f=u_0^{1/2-1/d}\), we recover Sobolev's inequality.


\subsection{Statement of the results}

In this subsection, we state the main results of this paper. Let us start by listing the hypotheses.

Let \(H\) be a strictly convex function from \(\R_+\) to \(\R\) such that \(H(0)=0\), and such that it is smooth on \(\R_+^*\). 
On \(\R^*_+\), define the functions \(\psi=H'\), \(U(x)=xH'(x)-H(x)\) and \(U_2(x)=xU'(x)-U(x)\). In everything that follows, we shall do the following hypothesis:
\begin{hypo}\label{hp:U}
 Assume that \( U_2+\frac{1}{d}U\geq 0.\)
\end{hypo}

Fix a closed convex set \(\overline\Omega\subset\R^d\), and choose a positive integrable function \(v\in C^\infty(\overline\Omega)\) such that
\begin{hypo}\label{hp:v}
  \(-\nabla^2\psi(v) \geq CI_d\) for some constant \(C>0\).  
\end{hypo}

\begin{thrm}\label{th:gen_sob}
  Let \(H\) be a strictly convex function from \(\R_+\) to \(\R\), such that \(H\) is smooth on \(\R^*_+\) and \(H(0)=0\), and define \(\psi=H'\). Fix a closed convex set \(\overline\Omega\subset\R^d\), and a smooth positive function \(v:\overline\Omega\to\R^*_+\). Under hypotheses \ref{hp:U} and \ref{hp:v}, for any positive function \(u\in \mc C^\infty(\overline\Omega)\) such that \(\int_\Omega u =\int_\Omega v\), the following inequality holds
  \begin{equation}\label{eq:gen_sob}
    \int_\Omega H(u)-H(v)-(u-v)\psi(v) \leq \frac{1}{2C}\int_\Omega u\norm{\nabla (\psi(u)-\psi(v))}^2.
  \end{equation}
\end{thrm}

\begin{rmrk}
  Note that hypothesis~\ref{hp:U} is related to the hypothesis leading to McCann's displacement convexity~\cite{mcc94}: indeed, it is equivalent to asking that \(x\mapsto x^dH(x^{-d})\) be a convex function of \(x\). This is not a surprise, since the method we develop here relies, as we shall see, on the convexity of a functional along a certain path in the Wasserstein space. However, this path is \emph{not} McCann's geodesic. Interestingly, the geodesics themselves can be used to prove the Sobolev inequality~\cite{cnv04}, and it is not clear why two different paths can be used to prove the same result, using the same condition on the functional.

 As a consequence of the similarity between those hypothesis, though, concrete applications of this method beyond Rényi entropies are still lacking, just like they are for displacement convexity.
\end{rmrk}

Theorem~\ref{th:gen_sob} may be seen as an immediate corollary of the (slightly) more general theorem that follows, where we allow \(v\) to take the value zero. However, we choose to present the two theorems separate, since theorem \ref{th:gen_sob} feels a bit more natural, it being easy to relate to a gradient flow, as will be seen in section \ref{se:heart}. To formulate this more general version, we first need to define the generalized inverse of a function.
\begin{defi}\label{def:gen_inverse}
  Let \(\psi:\R_+^*\to\R\) be a continuous strictly increasing function. Its generalized inverse \(\psi^{-1*}\) is given for \(x\in\overline\R\) by  
  \begin{equation}\label{eq:gen_inverse}
    \psi^{-1*}(x)=
    \begin{cases}
      \psi^{-1}(x) & \text{if }x\in\pare{\psi(0^+),\psi(+\infty)}\\
      0  & \text{if }x \leq \psi(0^+)\\
      +\infty & \text{if }x \geq \psi(+\infty).
    \end{cases}
  \end{equation}
\end{defi}

Instead of considering a smooth function \(v\), we instead look at the generalized inverse of some convex function \(V\), or, in other words, \(v=\psi^{-1*}(-V)\).  Note that \(v\) may very well be not differentiable, even if \(H\) and \(V\) are smooth. Also, since we do not want the function \(v\) to take the value \(+\infty\), as nothing would be integrable anymore. We thus replace hypothesis \ref{hp:v} by the following
\begin{hypo}\label{hp:V}
  \(-V < \psi(+\infty)\), and \(\nabla^2 V\geq CI_d\) for some constant \(C>0\).
\end{hypo}

\begin{thrm}\label{th:gen_sob_cc}
  Let \(H\) be a strictly convex function from \(\R_+\) to \(\R\), such that \(H\) is smooth on \(\R^*_+\) and \(H(0)=0\), and define \(\psi=H'\). Fix a closed convex set \(\overline\Omega\subset\R^d\), and a smooth function \(V:\overline\Omega\to\R\). Define \(v=\psi^{-1*}(-V)\), where \(\psi^{-1*}\) stands for the generalised inverse of \(\psi\), as defined in \eqref{eq:gen_inverse}.
  
  Under hypotheses \ref{hp:U} and \ref{hp:V}, for any positive function \(u\in \mc C^\infty(\overline\Omega)\) such that \(\int_\Omega u =\int_\Omega v\), the following inequality holds
  \begin{equation}\label{eq:gen_sob_cc}
    \int_\Omega H(u)-H(v)+(u-v)V \leq \frac{1}{2C}\int_\Omega u\norm{\nabla \psi(u)+\nabla V}^2.
  \end{equation}
\end{thrm}

\begin{rmrk}
  In this whole paper, we consider functions with compact support in a convex domain \(\overline\Omega\), but \(\overline\Omega\) will always be closed, which means that the functions are not necessarily equal to zero on \(\partial\Omega\). This is of special importance, because we use theorems \ref{th:gen_sob} and \ref{th:gen_sob_cc} to prove trace inequalities.
\end{rmrk}

The proof to theorems~\ref{th:gen_sob} and~\ref{th:gen_sob_cc} is rather long, and so will be split into two sections: section~\ref{se:heart} contains the somewhat formal but accurate calculations, and section~\ref{se:pde} addresses all the technicalities required to make the calculations rigorous. Among various Sobolev inequalities that may be proved using these results, two are, up to our knowledge, new and of particular interest.

\begin{coro}[Trace logarithmic Sobolev inequality]\label{th:trace_ls}
  For all \(h\in\R\), and for all positive functions \(u\in C^\infty(\R^d_+)\) such that \(\int_{\R^d_+} u=1\), the following inequality holds
  \begin{equation}\label{eq:ls_full}
    \int_{\R^d_+}u\log u \leq \frac{d}{2}\log\pare{\frac{1}{2\pi de}\int_{\R^d_+}\frac{\norm{\nabla u}}{u}^2} - \log{\gamma(\R^d_{+he})}-h\pare{\int_{\partial \R^d_+}u}\pare{\frac{1}{d}\int_{\R^d_+}\frac{\norm{\nabla u}}{u}^2}^{\!-1/2},
  \end{equation}
  where \(\gamma\) stands for the standard Gaussian probability measure. Furthermore, there is equality when \(u=C_h \exp\big(-\norm{x+he}^2\big)\), where \(C_h\) is chosen such that \(\int_{\R^d_+}u=1\).
\end{coro}
Note that for \(h=0\), this is the standard optimal logarithmic Sobolev inequality on the half space. Interestingly, the parameter \(h\) can be chosen either positive or negative, allowing the trace term to be used as an upper or a lower bound.

\begin{coro}[GNS inequality]\label{th:gns}
  Let \(p\in(0,1)\).
  For all functions \(f\in C^\infty(\R^d_+)\), the following inequality stands
  \begin{equation}\label{eq:gns}
    \norm{f}_{p+1}\leq C_p \norm{\nabla f}_2^\theta\norm{f}_{2p}^{1-\theta},
  \end{equation}
  where
  \[
  \theta=\frac{d(1-p)}{(1+p)(d(1-p)+2p)}.
  \]
  Furthermore, there is equality when \(f= \big(1-\norm x^2\big)_+^{1/(1-p)}\), up to multiplication by a constant, rescaling, and translation by a vector in \(\R^{d-1}\times\{0\}\).
\end{coro}
This is to say that the inequlaity on the half-plane is the same as the one on the whole space from del Pino and Dolbeault's paper~\cite{dpd02}, only with a different constant. Note that we focus here on the case \(p<1\), but the case \(p>1\) in theorem~\ref{th:dpd_gns} is aso a direct consequence of theorem~\ref{th:gen_sob}.
This result is actually a special case of the more general trace inequality \eqref{eq:trace_gns} that we will prove in section \ref{se:equiv}.


\section{Formal proof}\label{se:heart}

\subsection{Some words on \titlemath{$\Gamma$}{Gamma}-calculus}

As stated in the introduction, we choose in this article to stick to the Gamma calculus formalism (see \cite{bgl_book}) even though we do not study Markov semigroups. Let us very briefly introduce some notions here, which, in this particular case, are tied to the standard Laplacian \(\Delta=\sum_{i=1}^d\frac{\partial^2}{\partial x_i^2}\), but may very well be used with other diffusion operators, such as the Laplace-Beltrami operator on manifolds. 
\begin{defi}
  The carré du champ operator is the symmetric bilinear map from \(\mc C^\infty(\R^d)\times \mc C^\infty(\R^d)\) onto \(\mc C^\infty(\R^d)\) defined by
  \[
  \Gamma(a,b)=\frac{1}{2}(\Delta(ab)-a\Delta b-b\Delta a)=\nabla a\cdot\nabla b.
  \]
  Its iterated version is defined by
  \begin{align*}
    \Gamma_2(a,b)&=\frac{1}{2}(\Delta(\Gamma(a,b))-\Gamma(a,\Delta b)-\Gamma(b,\Delta a))\\
    &= \operatorname{tr}((\nabla^2a)^t\nabla^2b).
  \end{align*}
  Out of convenience, we will use the same notation for the bilinear maps and their respective quadratic maps, i.e.\ \(\Gamma(a)=\Gamma(a,a)\) and \(\Gamma_2(a)=\Gamma_2(a,a)\).
\end{defi}
With this formalism, the Hessian may be written in the following way: if \(f,\,g,\,h\) are smooth functions, then
\begin{equation}\label{eq:hessian_gamma}
  \nabla^2 f(\nabla g,\nabla h)=\frac{1}{2}\pare{\Gamma(g,\Gamma(f,h))+\Gamma(h,\Gamma(f,g))-\Gamma(f,\Gamma(g,h))}.
\end{equation}
A quick proof of this fact on manifolds can be found in \cite[Lemma 2.3]{GZ19}.

\begin{rmrk}
  The standard Laplacian on \(\R^d\) satisfies a \(CD(0,d)\) condition, or in other words
  \begin{equation}\label{eq:CD_laplacian}
    \Gamma_2(a)\geq\frac{1}{d}(\Delta a)^2
  \end{equation}
  for all smooth functions \(a\). This is nothing else than a Cauchy-Schwarz inequality, or a special case of the Bochner-Lichnerowicz inequality~\cite[Theorem C.3.3]{bgl_book}
\end{rmrk}
\begin{rmrk}
  In this article, we will consider functions defined on a closed convex subset \(\Omega\subset \R^d\). The definition of \(\Gamma,\,\Gamma_2\) trivially generalizes to such subsets.
  The major downside of using the \(\Gamma\) formalism is that the theory was not developped for functions taking nonzero values on the boundary of the domain, so instead of the usual neat integration by parts formula, we will have to use one adapted to our setting:
  \begin{align*}
    \int_\Omega \Gamma(a,b)&=-\int_\Omega b\Delta a + \int_{\partial\Omega} b\partial_\nu a\\
    &=-\int_\Omega a\Delta b + \int_{\partial\Omega} a\partial_\nu b,
  \end{align*}
  where \(\partial_\nu\) stands for the derivative along the outer normal vector.
\end{rmrk}


\subsection{Setting of the flow}

In this subsection, we assume that every function we manipulate is nice and smooth, and we rigorously prove a generalized version of inequality \eqref{eq:entropy_ineq1}, theorem \ref{th:second_derivative}, which is the key leading to theorem \ref{th:gen_sob}. We refer to section \ref{se:pde} for the technical study of the flow.

Fix some closed convex set \(\overline\Omega\in\R^d\), and some strictly convex smooth function \(H:\R^*_+\to\R\), and define \(\psi=H'\). Let \(v\in \mc C^\infty(\overline\Omega,\R_+^*)\) be a function such that
\begin{equation}\label{eq:v_convexity}
  -\nabla^2\psi(v)\geq CI_d
\end{equation}
for some positive constant \(C\).

\begin{rmrk}
  Note carefully that we assume here \(v\) to be positive, which is only true in theorem~\ref{th:gen_sob}. The rigorous proof of theorem~\ref{th:gen_sob_cc} will wait until section~\ref{se:pde}.
\end{rmrk}
\begin{rmrk}\label{rm:positivity}
  Again, since \(\Omega\) is closed, \(v\) is allowed to be nonzero at the boundary \(\partial\Omega\). This is important, since the typical example for function \(v\) is, just like in subsection \ref{se:model}, \(\psi^{-1}(a-\norm{x}^2)\), for some \(a\in\R\). Note that this inverse is not always well defined, and more generally, a positive function satisfying \eqref{eq:v_convexity} might not exist. We will come back to this in section \ref{se:pde}, as it will be of particular importance in the proof of theorem \ref{th:gen_sob_cc}.
\end{rmrk}

We consider the generalized entropy \(\mc F\) defined by
\begin{equation}\label{eq:entropy_def}
  \mc F(u)=\int_\Omega H(u)-u\psi(v).
\end{equation}
for positive smooth functions \(u\).
Since \(H\) is convex, \(\mc F(u)\geq \mc F(v)\) for any function \(u\). The idea in this section is to consider the entropy along the flow of this very entropy, namely
\begin{subequations}
  \label{eq:pde_system}
  \begin{alignat}{2}
    &\partial_tu=-\nabla\cdot(u\nabla\phi) && \quad\text{in}\  \R_+^*\times\Omega \label{eq:pde}\\
    &\partial_\nu\phi=0 &&\quad \text{in}\  \R_+^*\times\partial\Omega \label{eq:pde_boundary}\\
    &u(0,.)=u_0 &&\quad \text{in}\  \Omega. \label{eq:pde_initial}
  \end{alignat}
\end{subequations}
where \(u_0\in\mc C^\infty_c(\overline\Omega)\) is some positive initial data such that \(\int_\Omega u_0=\int_\Omega v\), and
\begin{equation}
  \phi=\psi(v)-\psi(u).
\end{equation}
Equation \eqref{eq:pde} is a generalized Fokker-Planck equation: indeed, whenever \(\psi=\log\) and \(v\) is the standard Gaussian, it is exactly a rewriting of the standard Fokker-Planck equation. We leave the technical study of this equation to section \ref{se:pde}, and assume for now that the solution to this problem not only exists at all times, is unique, but also that it is positive and smooth (at least smooth enough to do the calculations we are about to do, say \(\mc C^1\) with respect to time and \(\mc C^3\) with respect to space). As a first remark, we see that the \(L^1\) norm is preserved: using an integration by parts,
\begin{align*}
  \partial_t \int_\Omega u &= -\int_\Omega\nabla\cdot(u\nabla\phi)\\
  &=-\int_{\partial\Omega} u\partial_\nu\phi=0.
\end{align*}
Now, consider the entropy along the flow, which we write for brievity \(\Lambda(t)=\mc F(u(t,.))\) for \(t\geq0\). Differentiating the entropy with respect to time, we find, using an integration by parts,
\begin{align*}
  \Lambda'(t)&=\int \partial_tu(\psi(u)-\psi(v))\\
  &=\int \nabla\cdot(u\nabla\phi)\phi=-\int_\Omega u\Gamma(\phi)=-\mc I(u),
\end{align*}
the boundary term being zero due to the Neumann boundary condition. The reason behind the choice of the flow should now appear more clearly: the derivative of the entropy is, up to the sign, what is sometimes called the \emph{entropy creation} (or the generalized Fischer information) and written \(\mc I(u)\). More importantly, this shows that \(\Lambda'\) is nonpositive. Now, since the entropy decreases along the flow, and since \(v\) is the only global minimum of \(\mc F\) that has the same mass as \(u_0\),
it is reasonable to expect \(u\) to converge towards \(v\) in some sense as \(t\) goes to infinity, so we also assume that \(\lim_{t\to+\infty}\mc F(u)=\mc F(v)\). In the good tradition of the Bakry-Émery method, we may differentiate the entropy once more to find the following proposition.

\begin{prop}\label{th:second_derivative}
  The second derivative of the entropy along the flow of \eqref{eq:pde_system} is given by
  \begin{equation}\label{eq:second_derivative}
    \Lambda''(t)=2\int_\Omega -\nabla^2\psi(v)(\nabla\phi,\nabla\phi)u +(\Delta\phi)^2U_2(u)+\Gamma_2(\phi)U(u) - \int_{\partial\Omega}\partial_\nu\Gamma(\phi)U(u),
  \end{equation}
  where the functions \(U\) and \(U_2\) are given by
  \begin{equation}
    U(x)=xH'(x)-H(x),\text{ and }\ U_2(x)=xU'(x)-U(x),\,x\in\R^*_+.
  \end{equation}
\end{prop}
\begin{proof}
  Recall that \(\Lambda'(t)=-\int_\Omega u\Gamma(\phi)\). Let us differentiate this expression once more
  \begin{align*}
    \Lambda''(t)&=-\int_\Omega \pare{2u\Gamma(\phi,\partial_t\phi) +\partial_t u\Gamma(\phi)}\\
    &=-2\int_\Omega\Gamma(\phi,\partial_t\phi)u + \int_\Omega \nabla\cdot(u\nabla\phi)\Gamma(\phi) \\
    &=-2\int_\Omega\Gamma\left(\phi,\partial_t\phi+\frac{1}{2}\Gamma(\phi)\right)u  +\int_{\partial\Omega}u\Gamma(\phi)\partial_\nu\phi.
  \end{align*}
  The boundary term vanishes under the boundary condition \eqref{eq:pde_boundary}. Differentiating \(\phi=\psi(v)-\psi(u)\) with respect to time,
  \begin{align*}
    \partial_t\phi+\frac{1}{2}\Gamma(\phi)&=\nabla\cdot(u\nabla\phi)\psi'(u)+\frac{1}{2}\Gamma(\phi)\\
    &=\Gamma(\psi(u),\phi)+\frac{1}{2}\Gamma(\phi)+U'(u)\Delta\phi\\
    &=\Gamma(\psi(v),\phi)-\frac{1}{2}\Gamma(\phi)+U'(u)\Delta\phi.
  \end{align*}
  Now, on the one hand, applying equation~\eqref{eq:hessian_gamma} with \(f=\psi(v)\), \(g=h=\phi\), we find
  \begin{align*}
    -2\Gamma\pare{\phi,\Gamma(\psi(v),\phi)-\frac{1}{2}\Gamma(\phi)}u
    &=-2\nabla^2\psi(v)(\nabla\phi,\nabla\phi)u - \Gamma\big(\psi(v),\Gamma(\phi)\big)u +\Gamma\big(\phi,\Gamma(\phi)\big)u\\ 
    &= -2\nabla^2\psi(v)(\nabla\phi,\nabla\phi)u - \Gamma\big(U(u),\Gamma(\phi)\big),
  \end{align*}
  On the other hand,
  \begin{align*}
    \Gamma(\phi,U'(u)\Delta\phi)u&=\Gamma(\phi,\Delta\phi)U'(u)u+\Delta\phi\Gamma(\phi,U'(u))u \\
    &= \Gamma(\phi,\Delta\phi)U_2(u)) + \Gamma(\phi,\Delta\phi)U(u) + \Delta\phi\Gamma(\phi,U_2(u))\\
    &= \Gamma(\phi,U_2(u)\Delta \phi ) + \Gamma(\phi,\Delta\phi)U(u).
  \end{align*}
  We may now use integration by parts to find that
  \begin{multline}
    \Lambda''(t)=-2\int_\Omega\nabla^2\psi(v)(\nabla\phi,\nabla\phi)u + \int_\Omega U(u)\Delta\Gamma(\phi) - \int_{\partial\Omega}\partial_\nu\Gamma(\phi)U(u)  \\
    +2\int_\Omega(\Delta\phi)^2U_2(u) -2\int_{\partial\Omega}U_2(u)\Delta\phi\partial_\nu\phi-2\int_\Omega\Gamma(\phi,\Delta\phi)U(u),
  \end{multline}
  which concludes the proof, because \(\Gamma_2(\phi)=\frac{1}{2}\Delta\Gamma(\phi)-\Gamma(\phi,\Delta\phi)\), and because the second boundary term is zero.
\end{proof}

\(\bullet\) Now, differentiating the boundary condition \eqref{eq:pde_boundary}, and multiplying by \(\nabla\phi\), we find that
\[
0=\nabla^2\phi(\nabla\phi,\nu)+\nabla\nu(\nabla\phi,\nabla\phi)=\frac{1}{2}\partial_\nu\Gamma(\phi)+\nabla\nu(\nabla\phi,\nabla\phi),\,\text{on }\partial\Omega
\]
which, since \(\Omega\) is convex, implies that \(\partial_\nu\Gamma(\phi)\) is nonpositive.

\(\bullet\) By convexity, and since \(H(0)=0\), we know that \(U\geq0\), so that the boundary term is nonpositive.

\(\bullet\) Next, we may use the fact that the Laplacian on \(\R_d\) satisfies the \(CD(0,d)\) curvature-dimension condition \eqref{eq:CD_laplacian}, which implies that 
\[
(\Delta\phi)^2U_2(u)+\Gamma_2(\phi)U(u)\geq (\Delta\phi)^2\pare{U_2(u)+\frac{1}{d}U(u)}.
\]
Assume that this last term is nonnegative, and recall that we chose \(v\) so that \(-\nabla^2\psi(v)\geq CI\), so we may now claim that
\begin{equation}\label{eq:entropy_ineq}
  \Lambda''(t)\geq 2C\int u\Gamma(\phi) = -2C\Lambda'(t).
\end{equation}
With this inequality, we are now able to prove the following theorem:

\begin{thrm}\label{th:entropy}
  For all \(u_0\in\mc C^\infty_c(\overline\Omega)\) such that \(\int_\Omega u_0=\int_\Omega v\), the following inequality stands:
  \begin{equation}\label{eq:entropy}
    \mc F(u_0)-\mc F(v) \leq \frac{1}{2C} \mc I(u_0).
  \end{equation}
\end{thrm}
\begin{proof}
  The work is essentially done with proposition~\ref{th:second_derivative}, inequality \eqref{eq:entropy_ineq} being the heart of the now classical Bakry-Émery method. Integrating inequality \eqref{eq:entropy_ineq} between \(0\) and \(t\), we find that
  \[
  -\Lambda'(t)\leq -\Lambda'(0)e^{-2Ct}
  \]
  and then once again, between \(t=0\) and \(t=+\infty\), yields
  \[
  \Lambda(0)- \lim_{t\to+\infty}\Lambda(t)\leq -\frac{1}{2C} \Lambda'(0).
  \]
  Now, recall that \(\Lambda'(0)=-\mc I(u_0)\), and further assume that \(\lim_{t\to+\infty}\mc F(u)=\mc F(v)\) to conclude.
\end{proof}
The assumption of convergence we made on the entropy will be rigorously proved in section~\ref{se:pde}. Nevertheless, we insist that it is a behavior naturally expected: indeed, the derivative of the entropy is strictly negative whenever \(\nabla\phi\neq 0\), and \(u=v\) is the only function verifying both \(\nabla\phi=0\) and \(\int u=\int v\), so mass preservation must imply this convergence.

\begin{rmrk}
  Even though we fixed the value of \(H\) at \(0\), theorem~\ref{th:entropy} is invariant under summation of \(H\) with a constant: if it is true for \(H\), it remains true for \(H+C\), where \(C\in\R\). However, while \(\psi\) is invariant under this operation, \(U\) \emph{is not}, and becomes \(U-C\). This invariance property is recovered in equation~\eqref{eq:second_derivative} with the help formula
  \begin{equation}
    2\pare{\int_\Omega \Gamma_2(\phi) - (\Delta\phi)^2} -\int_{\partial\Omega}\partial_\nu\Gamma(\phi) =0.
  \end{equation}
\end{rmrk}


\subsection{Equivalent formulations of the entropy inequality}\label{se:equiv}

As has already been seen in subsection \ref{se:model}, inequality \eqref{eq:entropy} is completely equivalent to Sobolev's inequality when \(\Omega=\R^d\), and with
\begin{align*}
  H(x)&=-x^{1-1/d},                       &  V(x)&=1+\norm{x}^2,\\
  \psi(x)&=H'(x)=-\frac{d-1}{d}x^{-1/d},  &  v(x)&=\psi^{-1}(-V)=\pare{\frac{d-1}{d}\big(1+\norm x^2\big)}^{-d}.
\end{align*}
The Sobolev inequality being a limit case of the GNS inequality, it turns out that just changing the exponant in the definition of \(H\) leads to the whole family. Indeed, inequality~\eqref{eq:entropy} with \(H(x)=-x^\alpha/\alpha\) for some \(\alpha\in\pare{1-\frac{1}{d},1}\) readily implies the GNS inequality family mentioned in \cite{dpd02},with the help of proposition~\ref{th:dpd_gns}. The case \(H(x)=x^\alpha\) for \(\alpha>1\) is also considered in \cite{dpd02}, and may be proved just the same with theorem~\ref{th:gen_sob_cc}.

It is worth noting that the choice \(\overline\Omega=\R^d_+=\R^{d-1}\times\R_+\) and \(V(x)=a+\norm{x}^2\) implies that the normal derivative of \(V\) on \(\partial\Omega\) is \(\partial_\nu V(x)= 2x\cdot\nu=0\). This simple but important fact, as will be made clearer in the proof of theorem~\ref{th:trace_gns}, may then be used to prove sharp GNS or logarithmic Sobolev inequalities on \(\R^d_+\), following the exact same calculations as for the whole Euclidean space case. See for instance \cite{bcefgg}.

Following an idea in \cite{naz06}, we may choose \(V\) to be \(V(x)=a+\norm{x+e}^2\), where \(e\) is a constant vector in \(\R^d\). Bruno Nazaret succesfully used this idea to recover the sharp Sobolev inequality on the half space \(\R^d_+\), and has later been used to prove trace GNS inequalities on the half-space in \cite{bcefgg}, and on convex domains in \cite{zug17}. Again, it proves fruitful here, where theorem~\ref{th:gen_sob} leads to the same inequalities as those found in those articles in the \(p=2\) case. We will not prove them here, as the purpose of this article is not to be exhaustive, but will instead focus on two new inequalities, which proofs can be adapted for other inequalities.

We first turn to the proof of the trace logarithmic Sobolev inequality.
\begin{proof}[Proof of corollary \ref{th:trace_ls}]
  Fix \(h\in\R\), \(\overline\Omega=\R_+^d\), and let \(e\) be the \(d\)th unit vector, which is orthogonal to \(\partial\R_+^d\). Let
  \begin{align*}
    H(x)&=x\log(x)-x,                & V(x)&=\frac{1}{2}\norm{x+he}^2,\\
    \psi&=H'=\log,&  v(x)&=\psi^{-1}(\beta_h-V)=\frac{1}{C_h}e^{-\frac{1}{2}\norm{x+he}^2},
  \end{align*}
  where \(C_h\eqqcolon\exp(\beta_h)\) has been chosen so that \(\int_{\R^d_+} v=1\), or in other words, \(C_h=(2\pi)^{d/2}\gamma(\R^d_{+he})\), with \(\gamma\) being the standard Gaussian measure. With those choices, \(U(x)=x\), \(U_2(x)=0\), so that theorem \ref{th:gen_sob} applies with constant \(C=1\). For any nonnegative \(u\in C_c^\infty(\R^d_+)\) such that \(\int u=\int v=1\), the following inequality stands
  \[
  \int_{\R^d_+} H(u)-H(v)-(u-v)\psi(v) \leq \frac{1}{2}\int_{\R^d_+} u\norm{\nabla \psi(u)+\nabla V}^2.
  \]
  Notice first that \(v\psi(v)-H(v)=U(v)=v\), so that we are left with
  \[
  \int_\Omega u\log u - u\log(v) \leq \frac{1}{2}\int_\Omega \frac{\norm{\nabla u}}{u}^2 +\frac{1}{2}\int_\Omega u\norm{\nabla V}^2 + \int_\Omega \nabla V\cdot \nabla u.
  \]
  Now, noticing that \(\frac{1}{2}\norm{\nabla V}^2= V=-\log(C_hv)\), the respective second terms on the right and left-hand side simplify. We integrate by parts the last term to find
  \begin{equation}\label{eq:ls1}
    \begin{aligned}
      \int_{\R^d_+} u\log u & \leq \frac{1}{2}\int_{\R^d_+} \frac{\norm{\nabla u}}{u}^2  -\log(C_h)\int_{\R^d_+} u-\int_{\R^d_+} u\dive(x+he)  + h\int_{\partial{\R^d_+}} u\partial_\nu(x+he) \\
      & =   -d -\log(C_h)+\frac{1}{2}\int_{\R^d_+} \frac{\norm{\nabla u}}{u}^2 -h\int_{\partial{\R^d_+}} u.
    \end{aligned}
  \end{equation}
  The inequality we thus get is already a form of logarithmic Sobolev inequality, but we may go a little bit further to find a version that is similar to the standard inequalities. To do this, we rescale the function \(u\) and optimize with respect to the parameter. Indeed, inequality \eqref{eq:ls1} stays true when replacing \(u\) by \(u_\lambda = \lambda^du(\lambda\,.)\), so, for all \(\lambda>0\), we find that
  \begin{equation}\label{eq:ls2}
    \int_{\R^d_+} u\log u  \leq  -d -\log(C_h)-d\log(\lambda) + \frac{\lambda^2}{2}\int_{\R^d_+} \frac{\norm{\nabla u}}{u}^2  -h\lambda\int_{\partial{\R^d_+}} u.
  \end{equation}
  Now, we may choose for \(\lambda\) the value that minimizes the right-hand side of the inequality, but the resulting inequality is not pretty. Instead, we choose the \(\lambda\) that we would choose if \(h=0\), or, in other words, if there was no trace term and we were trying to prove the standard inequality. Hence, for
  \[
  \lambda=\pare{\frac{1}{d}\int_{\R^d_+}\frac{\norm{\nabla u}}{u}^2}^{-\frac{1}{2}},
  \]
  inequality \eqref{eq:ls2} turns into inequality \eqref{eq:ls_full} and corollary \ref{th:trace_ls} is proved. Note that for \(u=v\) all the inequalities are, in fact, equalities, which proves optimality.
\end{proof}
\begin{rmrk}
  Another version of a trace logarithmic Sobolev inequality has been found independantly in \cite{bcefgg} using optimal transport and an improved Borell-Brascamp-Lieb inequality.
\end{rmrk}
\begin{rmrk}
  Note that while we studied the case of \(\overline\Omega=\R^d_+\), the proof can immediately be extended to convex cones, much like in \cite{zug17}. Writing \(\overline\Omega\) as the epigraph of the convex function \(\vphi\), the trace term would then become \(\int_{\R^{d-1}} u(x,\vphi(x))dx\).
\end{rmrk}

Instead of proving corollary~\ref{th:gns}, we instead showcase the method in a slightly more general case. In particular, the result showcases, just like for the logarithmic Sobolev inequality, the ease with which trace inequalities may be recovered.
\begin{thrm}\label{th:trace_gns}
  Let \(p\in(0,1)\). For all \(h\in\R\), and for all positive \(w\in C^\infty(\R^d_+)\), the following inequality stands
  \begin{equation}\label{eq:trace_gns}
    \norm{w}_{L^{1+p}(\R^d_+)}^{1+p} \leq \sbra{ a_h\norm{\nabla w}_{L^2(\R^d_+)}\norm{w}_{L^{2p}(\R^d_+)}^{p} -hb_h\norm{w}_{L^{1+p}(\partial\R^d_+)}^{1+p}}\pare{\frac{\norm{w}_{L^{2p}(\R^d_+)}}{\norm{\nabla w}_{L^2(\R^d_+)}}}^{\!\!1/\delta}.
  \end{equation}
  Furthermore, there is equality whenever \(w(x)=(\beta_h-\norm{x+he}^2)_+^{1/(1-p)}\), where \(\beta_h\) is such that
  \[
  \int_{\R^d_+} \sbra{ \pare{\frac{1-p}{2p}}\pare{\beta_h-\norm{x+he}^2} }^{2p/(1-p)}dx = 1.
  \]
\end{thrm}
\begin{rmrk}
  The exact expression of the positive constants \(a_h\) and \(b_h\) is, in our opinion, too complicated to be made explicit in the theorem; we refer to the proof of theorem~\ref{th:trace_gns} and remark~\ref{rm:constants} instead.
\end{rmrk}
\begin{proof}
  To prove this inequality, we use the Rényi entropy with power not \(p\), but \(2p/(1+p)\). Thus, fix \(\alpha=2p/(1+p)>1\), \(h\in\R\), \(\Omega=\R_+^d\), let \(e\) be the \(d\)th unit vector. Then, consider
  \begin{align*}
    H(x)&=\frac{x^\alpha}{\alpha(\alpha-1)}, & V(x)&=\norm{x+he}^2,\\
    \psi(x)&=H'(x)=\frac{x^{\alpha-1}}{\alpha-1}, &  v(x)&=\psi^{-1*}(\beta_h-V)=\big((\alpha-1)(\beta_h-\norm{x+he}^2)\big)_+^{1/(\alpha-1)},
  \end{align*}
  where, again, \(\beta_h\) has been chosen so that \(\int_{\R^d_+}v=1\). In that case, \(U(x)=(\alpha-1)H(x)\) and \(U_2(x)=(\alpha-1)^2H(x)\geq 0\), so that, again, theorem \ref{th:gen_sob_cc} applies: for all nonnegative \(u\in \mc C^\infty(\R_+^d)\) such that \(\int u=1\),
  \[
  \int_{\R^d_+} H(u)-H(v)+(u-v)V \leq \frac{1}{4}\int_{\R^d_+} u\norm{\nabla \psi(u)+\nabla V}^2.
  \]
  Expanding both sides, then doing an integration by parts and simplifying, yields
  \[
  A\int_{\R^d_+} u^\alpha \leq B_h - h\int_{\partial\R^d_+} u^\alpha + D\int_{\R^d_+} \norm{\nabla u^{\alpha-1/2}}^2,
  \]
  where \(A,\,B\) and \(D\) are positive constants given by
  \begin{align*}
    A&=\frac{1}{(\alpha-1)}+d,  &%
    B_h&=\beta_h + (\alpha-1)\int_{\R^d_+} vV,  &%
    D&=\frac{\alpha}{(2\alpha-1)^2}.
  \end{align*}
  This inequality holding for any function of unit mass, we may, just like in the proof of theorem \ref{th:trace_ls}, rescale it with respect to a certain parameter. Replacing \(u\) by \(u_\lambda=\lambda^du(\lambda\,.)\) for \(\lambda > 0\), we find that
  \begin{equation}\label{eq:trace_gns_lambda}
    A\int_{\R^d_+} u^\alpha \leq B_h\lambda^{-\delta+1} - h\lambda\int_{\partial\R^d_+} u^\alpha + D\lambda^{\delta+1}\int_{\R^d_+} \norm{\nabla u^{\alpha-1/2}}^2,
  \end{equation}
  where \(\delta=d(\alpha-1)+1>1\). All the inequalities of this family are still, of course, optimal, since one implies all the others through rescaling. To get a more compact inequality, we may write it for a well-chosen \(\lambda\). An interesting choice could be to take the infimum of the right-hand side of equation~\eqref{eq:trace_gns_lambda} with respect to \(\lambda\), but as it turns out, the trace term complicates things a bit, and the resulting inequality is not the prettiest. Instead, we choose the \(\lambda\) that corresponds to the infimum of the right-hand side when \(h=0\), that is
  \[
  \lambda=\pare{\frac{B_h(\delta-1)}{D(\delta+1)\int \norm{\nabla u^{\alpha-1/2}}^2}}^{\frac{1}{2\delta}}.
  \]
  Inequality~\eqref{eq:trace_gns_lambda} then becomes
  \begin{equation}\label{eq:trace_gns_1}
    \int_{\R^d_+} u^\alpha \leq a_h\norm{\nabla u^{\alpha-1/2}}^{1-1/\delta}_{L^2(\R^d_+)}-hb_h\norm{\nabla u^{\alpha-1/2}}_{L^2(\R^d_+)}^{-1/\delta}\int_{\partial\R^d_+} u^\alpha,
  \end{equation}
  with the constants \(a_h\) and \(b_h\) given by
  \begin{align*}
    a_h &= \frac{B_h^{(\delta+1)/2\delta}D^{(\delta-1)/2\delta}}{A}\pare{ \pare{\frac{\delta+1}{\delta-1}}^{\!\frac{\delta-1}{2\delta}} + \pare{\frac{\delta-1}{\delta+1}}^{\!\frac{\delta+1}{2\delta}} }, \\
    b_h &= B_h^{1/2\delta}D^{-1/2\delta}\pare{\frac{\delta-1}{\delta+1}}^{\!1/2\delta}.
  \end{align*}
  We now go back to the same parameters as in theorem~\ref{th:dpd_gns}: rewriting inequality~\eqref{eq:trace_gns_1} with \(w=u^{\alpha-1/2}\) and \(p=1/(2\alpha-1)\in(0,1)\), we find that for all smooth positive functions \(w\) such that \(\norm w_{2p} = \norm{u}_{1} = 1\),
  \[
  \norm{w}_{L^{1+p}(\R^d_+)}^{1+p} \leq a_h\norm{\nabla w}_{L^2(R^d_+)}^{1-1/\delta} -hb_h\norm{\nabla w}_{L^2(R^d_+)}^{-1/\delta}\norm{w}_{L^{1+p}(\partial\R^d_+)}^{1+p}.
  \]
  Finally, removing the normalization \(\norm{w}_{2p}=1\), we find
  \[
  \norm{w}_{L^{1+p}(\R^d_+)}^{1+p} \leq \sbra{ a_h\norm{\nabla w}_{L^2(R^d_+)}\norm{w}_{L^{2p}(\R^d_+)}^{p+1/\delta} -hb_h\norm{w}_{L^{1+p}(\partial\R^d_+)}^{1+p}\norm{w}_{L^{2p}(\R^d_+)}^{1/\delta}}\norm{\nabla w}_{L^2(R^d_+)}^{-1/\delta},
  \]
  which proves inequality~\eqref{eq:trace_gns}, and yields corollary~\ref{th:gns} when applied to \(h=0\) (which we can do, since \(a_h\) and \(b_h\) are well-defined for all \(h\in\R\); we refer to remark~\ref{rm:constants} for further discussion on these constants).

  Furthermore, optimality being preserved throughout this development is a direct consequence of the fact that the final inequality~\eqref{eq:trace_gns} is invariant under multiplication by a constant, as well as rescaling. Going through the proof again, choosing \(u=v\) turns all the inequalities in equalities, proving that equality is reached in inequality~\eqref{eq:trace_gns} for a rescaling of \(v^{\alpha-1/2}\), and thus for \(v^{\alpha-1/2}\) itself.  
\end{proof}

\begin{rmrk}\label{rm:constants}
  The dependence of constants \(a_h\) and \(b_h\) in \(h\) is entirely contained in the dependence of \(\beta_h\) in \(h\), as the proof shows. However, \(\beta_h\) is, up to our knowledge, not explicit. One can easily get estimates of its value: for example, it is pretty clear that for any \(h\in\R\), \(\beta_h\geq \underline\beta>0\), where \(\underline\beta\) is such that
  \[
  \int_{\R^d} \sbra{ \pare{\frac{1-p}{2p}}\pare{\underline\beta-\norm{x}^2} }^{2p/(1-p)}dx = 1.
  \]
  This \(\underline\beta\) can be calculated using Euler's \(\Gamma\) function. What is more, one can see that for \(v\) take non-zero values whenever \(h>0\), a necessary condition is that \(\beta_h > h^2\). We could refine this analysis and prove that necessarily, \(\beta_h \sim h^2\) when \(h\) goes to \(+\infty\), but this would probably be outside of the scope of the present article.
\end{rmrk}

\begin{rmrk}
  Interestingly, trace GNS inequalities in the \(p>1\) case admit a slightly nicer formulation. This is made possible in the calculations because the constant \(B_h\) changes sign, and can then be absorbed by the gradient term using Young's inequality, which just so happens to maintain optimality \cite{bcefgg}.
\end{rmrk}


\section{Study of the degenerate parabolic PDE}\label{se:pde}

In this section, we fix some convex domain \(\overline\Omega\in\R^d\). Our goal is to show that the calculations we did in section \ref{se:heart} are valid. In this context, we are only interested in proving the entropy inequality \eqref{eq:entropy}, allowing us to make use of solutions to an approximated problem rather than the nontrivial system \eqref{eq:pde_system}. We propose a quick and (almost) self-contained proof of the entropy inequality \eqref{eq:entropy}. However, the study of solutions to the full problem is both relevant and delicate, and many open questions remain. We refer for instance to the work of \cite{cjmtu01}.

Equations \eqref{eq:pde} and \eqref{eq:pde_boundary} are not only nonlinear, but also degenerate. Equation \eqref{eq:pde} may be written
\[
\partial_tu =  \Delta U(u) + \text{l.o.t},
\]
where the function \(U\) is given by \(U(x)=x\psi(x)-H(x)\), as introduced in section \ref{se:heart}. We want to modify the function \(U\) in order to have both a lower and an upper bound on the parabolicity, so that the system falls in the scope of standard parabolic theory.

To that effect, for \(\eps>0\), we choose an approximation of \(U\), written \(U_\eps\), that coincides with \(U\) in the range \(\sbra{\eps,1/\eps}\). To regain parabolicity, we want \(U_\eps\) to be strictly increasing and affine outside of that range, but we also want it smooth, so we impose that \(U_\eps\) is affine in the range \(\R\backslash\sbra{\eps/2,\eps^{-1}+\eps}\) instead, as pictured on figure \ref{fig:eta}.
\begin{figure}[ht]
  \centering
  \includegraphics{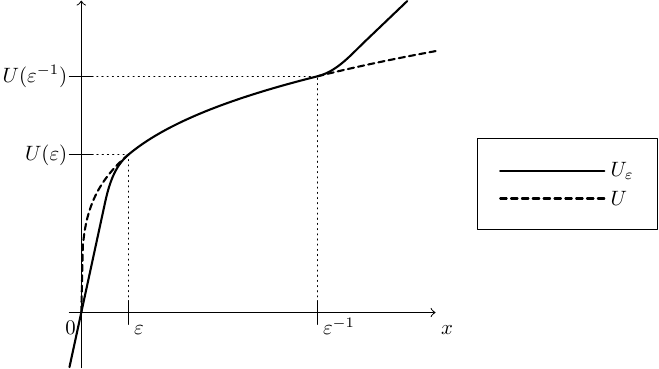}
  \caption{\(U_\eps\), an approximation of \(U\).}\label{fig:eta}
\end{figure}

From this choice of \(U_\eps\), and from the fact that \(U'(x)=x\psi'(x)\), we may also define \(\psi_\eps\) and \(H_\eps\) on \(\R_+^*\), by \(\psi_\eps(x)=\int_\eps^x \frac{U'_\eps(t)}{t} dt +\psi(\eps)\),  \(H_\eps(x)=x\psi_\eps(x)-U_\eps(x)\), so that they coincide respectively with \(\psi\) and \(H\) on the interval \(\sbra{\eps,1/\eps}\). With this definition, \(\psi_\eps\) is equivalent to a \(\log\) on \((0,\eps/2)\); for this reason, we use the function \(U_\eps\) in the formulation of the desingularized problem, because it is well-defined and smooth on the whole of \(\R\), which is needed if we want to directly apply the classical parabolic theory.

Thus, consider the problem \eqref{eq:pde_system} in which we replace \(U\) with \(U_\eps\)
\begin{subequations}
  \label{eq:pde_eps_system}
  \begin{alignat}{2}
      &\partial_tu=\Delta U_\eps(u)-\nabla\cdot(u\nabla\psi_\eps(v)) &&\quad \text{in}\  \R_+^*\times\Omega, \label{eq:pde_eps}\\
      &-\partial_\nu U_\eps(u)+u\partial_\nu\psi_\eps(v)=0 &&\quad \text{in}\  \R_+^*\times\partial\Omega, \label{eq:pde_eps_boundary}\\
      &u(0,.)=u_0 &&\quad \text{in}\  \Omega. \label{eq:pde_eps_initial}
  \end{alignat}
\end{subequations}


\subsection{Study of the desingularized problem}\label{se:pb_study}

\begin{thrm}\label{th:lsu_blackbox}
  Assume that \(\overline\Omega\) is smooth and bounded, and that \(u_0\) and \(v\) are smooth functions on \(\overline\Omega\) such that \(u_0\) verifies the compatibility condition \eqref{eq:pde_eps_boundary}. Then system \eqref{eq:pde_eps_system} admits a unique smooth solution on \(\Omega\times\R\).
\end{thrm}
This is the only classical result we invoke, and we will not prove it. Its proof can be found in \cite[Theorem 7.4, p. 491]{LSU_book}. Even though there exists versions of comparison principles in~\cite{LSU_book}, we formulate our own here. Let us first define subsolutions and supersolutions.

\begin{defi}
  Let \(u_1\) (resp. \(u_2\)) be a smooth function defined on \(\R_+\times\overline\Omega\). We say that \(u_1\) is a subsolution (\(u_2\) is a supersolution) of \eqref{eq:pde_eps_system} if for all time \(t\geq 0\),
  \begin{equation}\label{eq:ss_sol}
    \begin{cases}
      \partial_tu_1\leq\Delta U_\eps(u_1)-\nabla\cdot(u_1\nabla\psi_\eps(v)) & \text{in}\ \Omega\\
      -\partial_\nu U_\eps(u_1)+u_1\partial_\nu\psi_\eps(v)\geq 0 & \text{in}\ \partial\Omega,\\
    \end{cases}\ \text{ and }\ 
    \begin{cases}
      \partial_tu_2\geq\Delta U_\eps(u_2)-\nabla\cdot(u_2\nabla\psi_\eps(v)) & \text{in}\ \Omega\\
      -\partial_\nu U_\eps(u_2)+u_2\partial_\nu\psi_\eps(v)\leq 0 & \text{in}\ \partial\Omega.
    \end{cases}
  \end{equation}
\end{defi}

\begin{rmrk}
  This definition and the following proposition are more general than we will need them, since we will only consider actual solutions of the system, but it doesn't require any additional work, so we might as well prove it.
\end{rmrk}

\begin{prop}[Comparison principle]\label{th:comp}
  If \(u_1\) is a subsolution and \(u_2\) is a supersolution to \eqref{eq:pde_eps_system} such that \(u_1\leq u_2\) at time \(t=0\), then \(u_1\leq u_2\) for all times \(t\geq 0\).
\end{prop}

\begin{proof}
  Let \(u_1\) and \(u_2\) be as in \eqref{eq:ss_sol}. Their time derivatives \(\partial_tu_1,\partial_t u_2\) are continuous functions on a compact with respect to the space variable, and thus bounded at all times, hence, by domination, the following quantities are well-defined and equal:
  \[
  \partial_t\int_\Omega (u_1-u_2)_+ = \int \partial_t (u_1-u_2)_+. 
  \]
  Next, for \(m\in\N^*\), choose \(\rho_m\) to be a (non decreasing) \(C^1\) function approximating \(\chi_{\R_+^*}\). For example, consider \(\rho_m(x)=\rho(mx)\), where
  \[
  \rho(x)=
  \begin{cases}
    0 & \text{if }x\leq 0\\
    -2x^3 +3x^2 & \text{if }x\in\pare{0,1}\\
    1 & \text{if }x\geq 1,
  \end{cases}
  \]
  so that \(\norm{\rho_m'}_\infty=\frac{3}{2}m\). Using this approximation, we may write that
  \begin{equation}\label{eq:comp_approx}
    \partial_t\int_\Omega (u_1-u_2)_+ = \lim_{m\to+\infty} \int_\Omega \partial_t (u_1-u_2)\rho_m(Z),
  \end{equation}
  where \(Z\) can be any function such that \(Z(x)> 0\iff u_1(x)>u_2(x)\). We fix \(Z=U_\eps(u_1)-U_\eps(u_2)\). Since the function \(U_\eps\) is strictly increasing on \(\R\), such a \(Z\) constitutes a valid choice for equation~\eqref{eq:comp_approx}. Using~\eqref{eq:ss_sol} and integrating by parts, we find
  \begin{align*} 
    \int_\Omega \partial_t (u_1-u_2)\rho_m(Z) & \leq \int_\Omega\big(\Delta U_\eps(u_1)-\nabla\cdot(u_1\nabla\psi_\eps(v))-\Delta U_\eps(u_2)+\nabla\cdot(u_2\nabla\psi_\eps(v))\big)\rho_m(Z) \\
    &\leq\int_\Omega (-\nabla U_\eps(u_1)+\nabla U_\eps(u_2)+(u_1-u_2)\nabla\psi_\eps(v))\cdot(\rho_m'(Z)\nabla Z)\\
    &= \int_\Omega \pare{(u_1-u_2)\nabla Z\cdot\nabla\psi_\eps(v)-\Gamma(Z)}\rho'_m(Z)\\
    &\leq \frac{3}{2}m\int_{\{0<Z<1/m\}}\abs{u_1-u_2}\norm{\nabla Z}\norm{\nabla\psi_\eps(v)},
  \end{align*}
  since \(0\leq\rho'_m\leq 3m/2\), and \(\Gamma(Z)\geq0\). Finally, the mean value theorem applied to \(U_\eps\) yields
  \[
  \abs{u_1-u_2}\leq \norm{\frac{1}{U'_\eps}}_\infty \abs{U_\eps(u_1)-U_\eps(u_2)},
  \]
  which, applied to \(x\in \{ Z<1/m \}\), is enough to take the limit and conclude that
  \[
  \lim_{m\to+\infty}\int_\Omega \partial_t (u_1-u_2)\rho_m(Z)\leq 0,
  \]
  thereby concluding the proof.
\end{proof}

Let us now look into positive functions. If \(u>0\), we then write \(\phi_\eps=\psi_\eps(v)-\psi_\eps(u)\), and the equation \eqref{eq:pde_eps} takes the form
\[
\partial_tu=-\nabla\cdot(u\nabla\phi_\eps),
\]
allowing us to determine the positive stationary solutions. It is clear that \(v\) is one of them, and, more generally, all functions \(u\) such that \(\phi_\eps=\text{cst}\), are such solutions, and, as it turns out, they are the only ones. Indeed, if \(u\) is such a solution, testing equation \eqref{eq:pde_eps} against \(\phi_\eps\), and then integrating by parts and using~\eqref{eq:pde_eps_boundary}, we find
\begin{align*}
  0&=-\int_\Omega \phi_\eps\nabla\cdot(u\nabla\phi_\eps)\\
  &= \int_\Omega u \Gamma(\phi_\eps).
\end{align*}
Furthermore, notice that, by definition, \(\psi_\eps(x)=a\log(x)+b\) for all \(x\in(0,\eps/2)\), and also for all \(x>\eps^{-1}+\eps\), but with different constants. Therefore, \(\psi_\eps\) is actually a bijection between \(\R_+^*\) and \(\R\), and we may define, for any \(\alpha\in\R\), the positive stationary solution
\begin{equation}\label{eq:stationary_sol}
  v_\alpha=\psi_\eps^{-1}\pare{\psi_\eps(v)+\alpha}.
\end{equation}
These functions, being solutions, are both super- and subsolutions; and for any constant \(C>0\), we can find \(\alpha_1<\alpha_2\) such that \(0<v_{\alpha_1}<C<v_{\alpha_2}\) everywhere in \(\overline\Omega\), thus giving a priori \(L^\infty\) bounds on positive solutions, as well as \(L^{-\infty}\) bounds, both uniform in time.


\subsection{Proof of the entropy inequality}\label{se:proof_ineq}

We will now prove the entropy inequality~\eqref{eq:entropy} for the approximated entropy \(\mc F_\eps\). To that effect, owing to theorem \ref{th:lsu_blackbox} we now know that the system \eqref{eq:pde_system} has a smooth solution, so that proposition~\ref{th:second_derivative} is valid for the desingularized entropy flow. From there, three facts remain to be shown to conclude the proof of theorem \ref{th:entropy}: we will prove that
\begin{enumerate}
\item \(-\nabla^2\psi_\eps(v)\geq CI\);
\item everywhere in \(\overline\Omega\),
  \begin{equation}\label{eq:curvature_condition}
    U_{\eps,2}(u)+\frac{1}{d}U_\eps(u)=\pare{\frac{1}{d}-1}U_\eps(u)+uU_\eps'(u)\geq 0;
  \end{equation}
\item  the entropy \(\mc F_\eps(u)\) converges to \(\mc F_\eps(v)\) when \(t\to+\infty\).
\end{enumerate}
For the first point, we may assume that \(\eps\) has been chosen so that \(\eps\leq v\leq \eps^{-1}\) everywhere in \(\overline\Omega\). This implies that \(\psi_\eps(v)=\psi(v)\), and trivially, \(\nabla^2\psi_\eps(v)\leq-CI\).

The second point boils down to the construction of \(U_\eps\). We have assumed that \(U_{2}(u)+\frac{1}{d}U(u)\geq 0\), so inequality \eqref{eq:curvature_condition} is of course satisfied whenever \(\eps<u<\eps^{-1}\). We also made it so that for all \(r<\eps/2\), \(U_\eps(r)=ar\) for some \(a>0\). Then \(U_{\eps,2}(r)=0\), which directly implies that inequality \eqref{eq:curvature_condition} is satisfied in that range, and the same argument works for the range \(r>\eps+\eps^{-1}\). It thus suffices to show that inequality \eqref{eq:curvature_condition} is satisfied in the ranges \(\pare{\eps/2,\eps}\) and \(\pare{\eps^{-1},\eps^{-1}+\eps}\). It turns out that the choice of the smooth connections can be made so that it is true: to convince oneself of this fact, notice that it suffices to choose a smooth nonnegative connection for the quantity \(U_{\eps,2}(u)+\frac{1}{d}U_\eps(u)\) on the interval \((\eps/2,\eps)\) (and also on the interval \((\eps^{-1},\eps^{-1}+\eps)\)) and then use the following identity to recover~\(U_\eps\)
\[
\pare{\frac{1}{d}-1}U_\eps(x) + xU'_\eps(x) = x^{2-1/d}\pare{x^{-1+1/d}U_\eps(x)}',
\]
which also guarantees that \(U_\eps'>0\). Finally, we prove the following lemma:
\begin{lemm}
  If \(\int_\Omega u_0=\int_\Omega v\), then \(u\) converges towards \(v\) almost everywhere, and
  \[
  \lim_{t\to+\infty}\mc F_\eps(u)=\mc F_\eps(v).
  \]
\end{lemm}
\begin{proof}
  The comparison principle \ref{th:comp} ensures that there exists constants \(0<m<M\) such that \(m\leq u\leq M\) for all \((x,t)\in \overline\Omega\times\R_+\). Recall the proof of theorem~\ref{th:entropy}, we showed that
  \[
  0\leq \mc I_\eps(u) \leq e^{-2Ct}\mc I_\eps(u_0),
  \]
  so it is clear that \(\lim_{t\to+\infty}I_\eps(u)=0\), which readily implies that \(\lim_{t\to+\infty}\norm{\nabla \phi_\eps}_2=0\), since \(\mc I_\eps(u)=\int_\Omega u\Gamma(\phi_\eps)\geq m \int_\Omega\Gamma(\phi_\eps)=m\norm{\nabla\phi_\eps}_2^2\).
  The fact that \(u\) is uniformly bounded on \(\Omega\times\R_+\) implies that \(\phi_\eps\) is, too. Thus, \(\phi_\eps\) is uniformly bounded in \(H^1(\Omega)\), and we may extract a sequence of real numbers \((t_k)_{k\in\N}\) such that \(\phi_\eps\vert_{t=t_k} \rightharpoonup \phi_*\) weakly in \(H^1(\Omega)\). By weak lower semicontinuity, \(\norm{\nabla \phi_*}_2\leq \liminf_{k\to+\infty} \norm{\phi_\eps\vert_{t=t_k}}_2=0\), so that \(\phi_*\) is in fact a constant.

  Now, since \(u\) is also, in fact, bounded in \(H^1(\Omega)\), we may, without loss of generality, assume that \(u\vert_{t=t_k}\) converges almost everywhere to some function \(u_*\). By uniqueness of the limit,
  \[
  \phi_*=\psi(v)-\psi(u_*),
  \]
  so that \(u_*\) is actually one of the positive stationary solutions of \eqref{eq:pde_system} defined in equation \eqref{eq:stationary_sol}. But the fact that the flow is mass-preserving, combined with the dominated convergence theorem, implies that
  \[
  \int_\Omega u_* = \int_\Omega u_0 = \int_\Omega v,
  \]
  but the only stationary solution that has the same mass as \(v\) is \(v\) itself, so that \(u_*=v\). Indeed,
  \[
  \frac{d}{d\alpha}v_\alpha=\frac{1}{\psi_\eps'\circ\psi_\eps^{-1}(\psi_\eps(v)+\alpha)} >0.
  \]
  Finally, we may conclude that \(u\) converges almost everywhere to \(v\) as \(t\to+\infty\), and invoking, once again, dominated convergence, \(\lim_{t\to+\infty}\mc F_\eps(u)=\mc F_\eps(v)\).
\end{proof}

At this stage, we have proved the following: there exists \(\eps_0>0\), depending only on \(v\) and \(\Omega\), such that for all \(\eps\in(0,\eps_0)\),
\begin{equation}\label{eq:entropy_approx}
  \mc F_\eps(u_0)-\mc F_\eps(v) \leq \frac{1}{2C} \mc I_\eps(u_0)
\end{equation}
for all \(u_0\in \mc C^\infty_c(\overline\Omega)\), provided that \(\int_\Omega u_0=\int_\Omega v\) and that \(u_0\) satisfies the approximated compatibility condition \eqref{eq:pde_eps_boundary}. Now, fix some positive smooth function \(u_0\) satisfying the regular compatibility condition \eqref{eq:pde_boundary}, and that has the same mass as \(v\), then fix any \(0<\eps<\min(\eps_0,\min(u_0))\). By construction, the approximated entropy of \(u_0\) is then the same as the regular entropy of \(u_0\), and the same goes for the entropy production, so inequality \eqref{eq:entropy_approx} is valid, and is identical to inequality~\eqref{eq:entropy}.


\subsection{Extension to convex domains and generic smooth positive functions}\label{se:extension}

We have now proved that the entropy inequality
\begin{equation}\label{eq:entropy2}
  \mc F(u_0)-\mc F(v) \leq \frac{1}{2C} \mc I(u_0)
\end{equation}
holds true for smooth and positive functions \(u_0\) defined on a compact, convex and smooth set, as long as they verify the compatibility condition \eqref{eq:pde_boundary}, and that they have the same mass as the function \(v\). Working on the inequality \eqref{eq:entropy2} rather than the partial differential equation \eqref{eq:pde_system}, we may generalize this result by lifting the constraints.

Let us first extend the class of functions for which inequality \eqref{eq:entropy2} holds true. Let \(\overline\Omega\subset \R^d\) be compact, convex and smooth, and let \(u\) be a positive and smooth function defined on \(\overline\Omega\). We want to construct \(\tilde u=u+g\), an approximation of \(u\) that verifies the compatibility condition \eqref{eq:pde_boundary}: on \(\partial\Omega\),
\begin{align*}
    \partial_\nu(\psi(v)-\psi(\tilde u))&=0 \\
  \iff  \partial_\nu(g + u)\psi'(g + u)& = \partial_\nu\psi(v),
\end{align*}
and it is thus sufficient to find a function \(g\), small in some sense, such that, on \(\partial\Omega\),
\begin{align}
  g&=0, \label{eq:g0}\\
  \partial_\nu g&=\frac{\partial_\nu\psi(v)}{\psi'(u)}-\partial_\nu u=\frac{\nabla(\psi(v)-\psi(u))}{\psi'(u)}\cdot\nu. \label{eq:gprime0}
\end{align}
In dimension 1, the construction is somewhat straightforward. Assume just for now that \(\overline\Omega=\R_+\). The problem reduces to finding a reasonably small function that is zero on \(\partial \R_+=\{0\}\), and that has an assigned slope at that same point. We thus construct a function that looks like a small ridge: choose \(\eta\), defined on \(\R_+\), such that it is smooth, has compact support in \(\sbra{0,3}\), and is equal to identity on \(\sbra{0,1}\), like pictured on figure \ref{fig:eta2}. Then, the function \(x\mapsto C\delta\eta(x/\delta)\), where \(C\) is the desired slope at zero, satisfies everything we need: its \(L^\infty\) norm tends towards zero when \(\delta\to0\), the \(L^\infty\) norm of its derivative is bounded, and its support is included in \(\sbra{0,3\delta}\).
\begin{figure}[ht]
  \centering
  \includegraphics{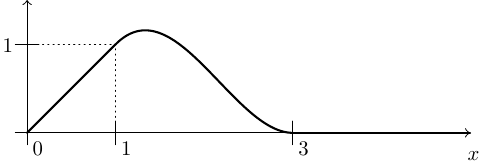}
  \caption{A ridge function \(\eta\).}\label{fig:eta2}
\end{figure}

Let us now return to more general \(\Omega\), and extend this construction. Note that since \(\Omega\) is smooth, its boundary admits a neighbordhood verifying the \emph{unique nearest point} property (for a short reference, see for instance \cite{foote}). In other words, there exists an open neighbordhood \(U\) of \(\partial\Omega\) and a smooth function \(P: U\to \partial\Omega\) such that for all \(x\in U\), \(d(x,\partial\Omega)=\norm{x-P(x)}\). This function \(P\) is called the projection onto \(\partial\Omega\), it is smooth, and its gradient at \(x\) is orthogonal to the tangent space at \(P(x)\). Thus, for all \(\delta\in(0,\delta_0)\), with \(\delta_0\) sufficiently small, the function
\[
g(x)=\delta\eta(d(x,\partial\Omega)/\delta)\pare{\frac{\partial_\nu\psi(v)}{\psi'(u)}-\partial_\nu u}(P(x))
\]
is well-defined, smooth, and satisfies both assumptions \eqref{eq:g0} and \eqref{eq:gprime0}. 
Furthermore, writing \(C_g=\pare{\int_\Omega u}\pare{\int_\Omega u+g}^{-1}\) and invoking the dominated convergence theorem, it is quite clear that \(\lim_{\delta\to0}\mc F(C_g(u+g))= \mc F(u)\), and \(\lim_{\delta\to0}\mc I(C_g(u+g))=\mc I(u)\), and thus the compatibility condition is lifted. Next, we want to further extend the result to more general domains. Let \(\overline\Omega\) be convex and compact and \(u\in C^\infty_c(\overline\Omega)\). The domain \(\overline\Omega\) may be approximated from within by smooth convex sets: 
\begin{lemm}
  For each \(\eps>0\), there exists \(\overline\Omega_\eps\subset\overline\Omega\) such that \(\overline\Omega_\eps\) is smooth and convex, and \(\abs{\Omega\backslash\Omega_\eps}<\eps\).
\end{lemm}
Using this fact, for any \(\eps>0\), inequality \eqref{eq:entropy2} holds for the restriction of \(C_\eps u\vert_{\Omega_\eps}\), where \(C_\eps\) is the normalisation constant \(C_\eps=\pare{\int_\Omega u}\pare{\int_{\Omega_\eps} u}^{-1}\). The dominated convergence theorem allows to take the limit when \(\eps\to 0\), proving inequality \eqref{eq:entropy2} for compact domains. We may finally extend the result for unbounded domains by considering \(\overline\Omega\cap B(0,R)\), where \(\overline\Omega\) is assumed to be closed and convex. Again, dominated convergence allows to take the limit \(R\to+\infty\), whence we proved theorem \ref{th:gen_sob} in its full generality.

\begin{proof}[Proof of the lemma]
  Let \(\overline\Omega\subset\R^d\) be compact and convex. Fix the distance function \(d_\Omega:x\mapsto d(x,\Omega)\), and choose some smoothing kernel \(\rho:\R^d\to\R_+\), such that \(\rho\in C^\infty_c(\R_d)\), and satisfiying \(\int \rho=1\) and \(B_1\subset\supp(\rho)\subset B_2\). Define, for \(\delta>0\), \(\rho_\delta=\delta^{-d}\rho(./\delta)\). The function \(d_\Omega*\rho_\delta\) is smooth, and also convex since \(\rho\geq 0\) and \(\Omega\) is convex. Now, notice that
  \[
  \{x\in\Omega, d(x,\partial\Omega)>2\delta\} \subset \{d_\Omega*\rho_\delta=0\} \subset \{x\in\Omega, d(x,\partial\Omega)>\delta\}.
  \]
  We now claim that there exists \(t_0>0\) such that \(\{d_\Omega*\rho_\delta<t_0\}\subset\Omega\). This is due to the continuity of \(d_\Omega*\rho_\delta\), and also the fact that \(\overline\Omega\) is compact. Now, by Sard's theorem, there exists a \(t\in(0,t_0)\) such that \(\{d_\Omega*\rho_\delta<t\}\) is smooth, and convex since it is a sublevel set of a convex function, and
  \[
  \{x\in\Omega, d(x,\partial\Omega)>2\delta\} \subset \{d_\Omega*\rho_\delta<t\} \subset \Omega.
  \]
  Finally, notice that \(\{x\in\Omega, d(x,\partial\Omega)>2\delta\} + B_{2\delta} = \mathring{\Omega} \), and thus, Brunn-Minkowski's inequality allows us to conclude that we may have chosen \(\delta\) small enough so that \(\abs{\Omega\backslash\{d_\Omega*\rho_\delta<t\}}<\eps\), which concludes the proof.
\end{proof}

\subsection{Generalized inverse}

In this subsection, we prove theorem \ref{th:gen_sob_cc}.
As mentioned in remark \ref{rm:positivity}, positive functions \(v\) satisfying \(-\nabla^2\psi(v)\geq CI_d\) might not alway exist. The natural example of when this is a problem is the flow related to the porous medium equation: when \(H(x)=\frac{x^\alpha}{\alpha(\alpha-1)}\), with \(\alpha>1\), then \(\psi\) is a one to one map from \(\R_+\) onto itself, and for any choice of \(a\in\R\), the function \(x\mapsto a-\norm x^2\) takes negative values. We would like to still make sense of this computation in that case.

Instead of fixing the function \(v\), choose a function \(V\in\mc C^\infty(\Omega,\R)\) such that its Hessian is bounded below, \(\nabla^2 V\geq CI_d\). While \(\psi^{-1}(-V)\) might not be well defined, we may consider, for \(\eps>0\), the function \(v_\eps=\psi_\eps^{-1}(-V)\). Recall that \(\psi_\eps\) behaves like a natural logarithm on a neighbourhood of zero, as well as towards infinity, so that \(v_\eps\) is well defined for all \(\eps>0\). Furthermore, \(v_\eps\in C^\infty(\Omega,\R^*_+)\). As \(\eps\) goes to \(0\), \(v_\eps\) converges to the so-called \emph{generalized inverse}  of \(\psi\), applied to \(-V\), which we will write \(\psi^{-1*}\):
\begin{itemize}
\item if \(\psi(0^+)<-V<\psi(+\infty)\), then it is clear that \(\lim_{\eps\to0}\psi_\eps^{-1}(-V) = \psi^{-1}(-V)\);
\item if \(-V\leq\psi(0^+)\), then, in particular, \(-V<\psi_\eps(\eps)\), and so \(0<\psi_\eps^{-1}(-V)<\eps\), so that \(\lim_{\eps\to0}\psi_\eps^{-1}(-V)=0\);
\item finally, if \(-V\geq\psi(+\infty)\), \(\psi_\eps^{-1}(-V)>1/\eps\), proving that \(\lim_{\eps\to0}\psi_\eps^{-1}(-V)=+\infty\).
\end{itemize}
Let \(v=\psi^{-1*}(-V)=\lim_{\eps\to 0} v_\eps\). Note that the function \(v\) is, in general, not even differentiable. For example, in the case where \(H(x)=x^2/2\) and \(V(x)=1-\norm{x}^2\), the generalized inverse of \(\psi(x)=x\) and the limit function \(v\) are given by
\[
\psi^{-1*}(x)=x_+,\quad v(x)=\pare{1-\norm x^2}_+.
\]
We do not really know how to make sense of the case where \(v\) is not finite, so we further assume that \(-V<\psi(+\infty)\) everywhere.

We may now fix an \(\eps>0\) and return to the previous subsections, where we replace the function \(v\) by the function \(v_\eps\). The study of the partial differential equation, subsection \ref{se:pb_study} remains unchanged, and the conclusions are the same. As far as subsection \ref{se:proof_ineq}, points \(2.\) and \(3.\) are unchanged too, and point \(1.\) is trivial: \(-\nabla^2\psi_\eps(v_\eps)=\nabla^2 V\geq CI_d\) by hypothesis, and the conclusion is still valid. If \(\overline\Omega\) is compact, convex and smooth, for all smooth positive functions \(u_0\) satisfying both \(\int_\Omega u_0 = \int_\Omega v_\eps\) and the compatibility condition \eqref{eq:pde_eps_boundary},
\[
\mc F_\eps(u_0) - \mc F_\eps(v_\eps) \leq \frac{1}{2C}\mc I_\eps(u_0).
\]
Again, following section \ref{se:extension} we may lift the compatibility condition, as well as the smoothness condition for \(\Omega\). We will tackle the boundedness only later, out of convenience. Let us write the entropy inequality fully. 
\begin{equation}\label{eq:entropy3}
  \int_\Omega H_\eps(u_0)-H_\eps(v_\eps)-(u-v_\eps)\psi_\eps(v_\eps) \leq \frac{1}{2C}\int_\Omega u_0\norm{\nabla\psi_\eps(u_0)-\nabla\psi_\eps(v_\eps)}^2.
\end{equation}
By construction, \(\psi_\eps(v_\eps)=-V\). Furthermore, for a fixed positive \(u_0\in \mc C^\infty(\Omega)\), we may choose \(\eps_0>0\) so that \(\eps_0<u_0<1/\eps_0\), and equation \eqref{eq:entropy3} rewrites
\begin{equation}
  \int_\Omega H(u_0)-H_\eps(v_\eps)+(u-v_\eps)V \leq \frac{1}{2C} \int_\Omega u_0\norm{\nabla\psi(u_0)+\nabla V}^2.
\end{equation}
for any \(0<\eps<\eps_0\). We just need to pass to the limit to prove theorem \ref{th:gen_sob_cc}. By Fatou's lemma,
\begin{align*}
  \int_\Omega H(u_0)-\liminf_{\eps\to 0}H_\eps(v_\eps)+(u-v)V &\leq  \liminf_{\eps\to0}\pare{\int_\Omega H(u_0)-H_\eps(v_\eps)+(u-v_\eps)V}\\
  &\leq \frac{1}{2C} \int_\Omega u_0\norm{\nabla\psi(u_0)+\nabla V}^2,
\end{align*}
so it suffices to show that \(\liminf_{\eps\to0}H_\eps(v_\eps) \leq H(v)\). Let \(x\in\overline\Omega\), we are faced with two cases, since we assumed that \(v\) is finite everywhere.
\begin{itemize}
\item If \(-V(x)\in(\psi(0^+),\psi(+\infty))\), then \(0<v(x)<+\infty\) and, since \(H_\eps\) coincides with \(H\) on the interval \(\sbra{\eps,1/\eps}\), it is clear that \(\lim_{\eps\to0} H_\eps(v_\eps(x))=H(v(x))\).
\item If \(-V(x)\leq \psi(0^+)\), then \(v_\eps(x)\to 0\), and since \(H_\eps(0)=0\),
  \[
  H_\eps(v_\eps(x))=\int_0^{v_\eps(x)}\psi_\eps(t)dt \leq v_\eps(x)\psi_\eps(v_\eps(x))= -v_\eps(x) V(x),
  \]
  so \(\lim_{\eps\to0} H_\eps(v_\eps(x))\leq 0=H(v(x))\).
\end{itemize}
This concludes the proof of theorem \ref{th:gen_sob_cc}. Notice that while we proved it for \(\mc C^\infty\) functions, it makes sense for the function \(u_0=v\) even though it is not necessarily differentiable.
Indeed, on the interior of \(\supp(v)\), \(v\) is smooth and \(\nabla\psi(v)=-V\). On the other hand, on the interior of \(\Omega\backslash\supp(v)\), \(\psi(v)\) is still well defined, because \(v\) can only be zero when \(-V\leq \psi(0^+)\), which means that \(\psi(0^+)\in\R\), and we may conclude that  \(v\norm{\nabla\psi(v)+V}^2=0\) on that set. In this sense, inequality \eqref{eq:gen_sob_cc} is optimal, because both sides are equal to \(0\) when \(u_0=v\).

\begin{rmrk}
  In the particular case of \(H(x)=\frac{x^\alpha}{\alpha(\alpha-1)}\), we use generalized inverses only when \(\alpha>1\). It just so happens that in that case, the function \(U(x)=\frac{x^\alpha}{\alpha}\) is convex, and thus \(U_2\geq 0\). This implies that to get inequality \eqref{eq:entropy_ineq} and ultimately to the entropy inequality theorem~\ref{th:gen_sob}, we only need a \(CD(0,\infty)\) assumption, and not any more the stronger \(CD(0,d)\) assumption. This does not matter so much in our case because we are only considering \(\R^d\), but it might prove useful on manifolds.  
\end{rmrk}





\bibliographystyle{alpha}
\bibliography{biblio-entropy}

\end{document}